\documentclass[12pt,twoside]{amsart}
\usepackage{amssymb,amsmath,amsthm, amscd, enumerate, mathrsfs}
\usepackage[all]{xy}
\usepackage{comment}

\title{On generic vanishing for pluricanonical bundles}
\author{Takahiro Shibata} 
\date{}
\address{Department of Mathematics, Graduate School of Science, 
Kyoto University, Kyoto 606-8502, Japan}
\email{tshibata@math.kyoto-u.ac.jp}

\newcommand{\Exc}[0]{{\operatorname{Exc}}}
\newcommand{\Supp}[0]{{\operatorname{Supp}}}
\DeclareMathOperator{\Gr}{Gr}
\DeclareMathOperator{\Pic}{Pic}
\DeclareMathOperator{\Spec}{Spec}
\DeclareMathOperator{\alb}{alb}
\DeclareMathOperator{\rank}{rank}
\DeclareMathOperator{\Coh}{Coh}
\DeclareMathOperator{\Ker}{Ker}
\DeclareMathOperator{\codim}{codim}
\DeclareMathOperator{\id}{id}

\DeclareMathOperator{\red}{red}
\DeclareMathOperator{\ord}{ord}
\DeclareMathOperator{\tor}{tor}

\newtheorem{thm}{Theorem}[section]
\newtheorem{lem}[thm]{Lemma}
\newtheorem{cor}[thm]{Corollary}
\newtheorem{prop}[thm]{Proposition}

\newtheorem{ques}[thm]{Question}

\theoremstyle{definition}

\newtheorem{rem}[thm]{Remark}
\newtheorem{ex}[thm]{Example}
\newtheorem*{ack}{Acknowledgments} 
\newtheorem{say}[thm]{}
\newtheorem{step}{Step}
\begin{document}

\begin{abstract}
We study cohomology support loci and higher direct images of (log) pluricanonical 
bundles of smooth projective varieties or log canonical pairs. 
We prove that the 0-th cohomology support loci of log pluricanonical bundles 
are finite unions of torsion translates of subtori,  
and we give a generalization of the generic vanishing theorem 
to log canonical pairs. 
We also construct an example of morphism from a smooth projective variety 
to an abelian variety such that a higher direct image of a pluricanonical bundle 
to the abelian variety is not a GV-sheaf. 
\end{abstract}

\maketitle

\tableofcontents 

\section{Introduction}\label{sec1}
Thoughout this paper, we always assume that all varieties are defined over the 
complex number field. 
Let $X$ be a smooth projective variety and $\Delta$ be a simple normal 
crossing divisor on $X$. 
In this paper, we prove some results of generic vanishing theory for $K_X+ 
\Delta$ or, more generally, $m(K_X+\Delta)$ for any positive integer $m$. 

In Section \ref{sec3}, we prove some results about the structure of  cohomology support loci for a log canonical pair $(X, \Delta)$ 
(for the definition of cohomology support loci, see \ref{say2.1}).  
Originally, Simpson proved the following theorem in \cite{Sim}: 

\begin{thm}[Simpson]\label{thm_Sim}
Let $X$ be a smooth projective variety. 
Then the cohomology support locus 
$$S^i_j(K_X) = 
\{ \xi \in \Pic ^0(X) \mid h^i(X, \mathcal O_X(K_X) \otimes \xi) \geq j \}$$
is a finite union of torsion translates of abelian subvarieties of $\Pic^0(X)$ 
for any $i \geq 0$ and $j \geq 1$. 
\end{thm}

In \cite{ClHa}, a generalization of Theorem \ref{thm_Sim} for Kawamata log terminal 
pairs is discussed (see \cite[Theorem 8.3]{ClHa}). 
We will generalize Theorem \ref{thm_Sim} for log canonical pairs: 

\begin{thm}[= Theorem {\ref{thm3.4}}]\label{thm1.1}
Let $X$ be a smooth projective variety, 
$\Delta$ be a boundary $\mathbb Q$-divisor on $X$, 
that is, a $\mathbb Q$-divisor whose coefficients are in $[0,1]$, 
with simple normal crossing support, 
$f:X \to A$ be a morphism to an abelian variety, 
and $D$ be a Cartier divisor on $X$ 
such that $D \sim _{\mathbb Q} K_X + \Delta$. 
Then the cohomology support locus
$$S^i_j(D, f) = \{ \xi \in \Pic ^0(A) \mid h^i(X, \mathcal O_X(D) \otimes f^*\xi) \geq j \}$$
is a finite union of torsion translates 
of abelian subvarieties of $\Pic ^0(A)$ 
for any $i \geq 0$ and $j \geq 1$. 
\end{thm}

We will reduce Theorem \ref{thm1.1} to the special case when $\Delta$ is a 
(reduced) simple normal crossing divisor. 
This case was treated by Kawamata \cite{Kaw} 
and was used in \cite{CKP} for the proof of Corollary \ref{thm_CKP} below. 
In \cite{Bud}, Budur proved a result for the cohomology support loci of unitary local systems, which includes it as a special case 
(see Theorem \ref{thm3.2} and Theorem \ref{thm3.3}). 
We will use Budur's result for the proof of Theorem \ref{thm1.1}. 

We can also prove a similar result for pluricanonical divisors:

\begin{thm}[= Theorem {\ref{thm3.7}}]\label{thm1.2}
Let $X$ be a normal projective variety, 
$\Delta$ be a boundary $\mathbb Q$-divisor on $X$ 
such that $K_X+\Delta$ is $\mathbb Q$-Cartier and $(X, \Delta)$ is 
log canonical, 
$f:X \to A$ be a morphism to an abelian variety, 
and $D$ be a Cartier divisor on $X$ 
such that $D \sim _{\mathbb Q} m(K_X + \Delta)$ for some positive integer $m$. 
Then the cohomology support locus
$$S^0_j(D, f) = \{ \xi \in \Pic ^0(A) \mid h^0(X, \mathcal O_X(D) \otimes f^*\xi) \geq j \}$$
is a finite union of torsion translates 
of abelian subvarieties of $\Pic ^0(A)$ 
for any $j \geq 1$. 
\end{thm}

Theorem \ref{thm1.2} is one of the main results of this paper. 
Note that Theorem \ref{thm1.2} states the structure of cohomology 
support loci of only 0-th cohomology 
and that $(X, \Delta)$ is an arbitrary log canonical pair.  
The special case when $\Delta=0$ was proved by Chen and Hacon (see 
\cite[Theorem 3.2]{ChHa}). 
As a corollary, Theorem \ref{thm1.2} implies the following result 
by Campana, Koziarz and P\u aun: 

\begin{cor}[{\cite[Theorem 0.1]{CKP}}]\label{thm_CKP}
Let $X$ be a normal projective variety, 
$\Delta$ be a boundary $\mathbb Q$-divisor on $X$ 
such that $K_X+\Delta$ is $\mathbb Q$-Cartier and 
$(X, \Delta)$ is log canonical. 
Assume that there exists a numerically trivial Cartier divisor $\rho$ on $X$ 
such that  
$H^0(X, m(K_X+\Delta)+\rho) \neq 0$ for some positive integer $m$. 
Then $h^0(X,m'(K_X+\Delta)) \geq h^0(X, m(K_X+\Delta) + \rho) >0$ 
for some suitable multiple $m'$ of $m$. 
\end{cor}

Therefore, we can see Theorem \ref{thm1.2} as a generalization of 
\cite[Theorem 0.1]{CKP}. 
We will also give an alternative proof of \cite[Corollary 3.2]{CKP} 
as an application of Theorem \ref{thm1.2} 
(see Proposition \ref{prop3.8} below). 

Section \ref{sec4} treats the higher direct images of pluricanonical bundles 
to abelian varieties. 
In this direction, Hacon first proved the following result in \cite{Hac}: 

\begin{thm}[Hacon]\label{thm_Hac}
Let $X$ be a smooth projective variety and $f: X \to A$ be a morphism 
to an abelian variety. 
Then the higher direct images $R^jf_*\omega_X$ are GV-sheaves for any $j$. 
\end{thm}

In the notations in Theorem \ref{thm_Hac}, a coherent sheaf $\mathcal F$ on 
$A$ is called a \textit{GV-sheaf} on $A$ if the cohomology support loci
$$ S^i(\mathcal F) = \{ \xi \in \Pic^0(A) \mid h^i(A, \mathcal F \otimes \xi) 
\neq 0 \}$$ satisfies $\codim_{\Pic^0(A)}S^i(\mathcal F) \geq i$ for every 
$i > 0$ (see also \ref{say2.3}). 
Theorem \ref{thm_Hac} recovers the original generic vanishing theorem of Green 
and Lazarsfeld in \cite{GrLa}. 
We will give some generalizations of Theorem \ref{thm_Hac} and the generic 
vanishing theorem of Green and Lazarsfeld to log canonical pairs. 
For the precise statements, 
see Proposition \ref{prop4.2} and Theorem \ref{thm3.5}, respectively.  

In \cite{PoSc}, Popa and Schnell gave a similar result for pluricanonical divisors: 

\begin{thm}[{\cite[Theorem 1.10]{PoSc}}]\label{thm_PoSc}
Let $X$ be a smooth projective variety 
and $f:X \to A$ be a morphism to an abelian variety.
Then the direct image $f_*(\omega_X^{\otimes m})$ 
is a GV-sheaf for any positive integer $m$. 
\end{thm}

Theorem \ref{thm_PoSc} is a consequence of \cite[Theorem 1.7]{PoSc}. 
For the proof of Theorem \ref{thm_PoSc}, the following special case of  
\cite[Theorem 1.7]{PoSc} is sufficient: 

\begin{thm}[{\cite[Corollary 2.9]{PoSc}}]\label{thm_PoSc2}
Let X be a smooth projective variety and $f: X \to Y$ be a morphism to 
a projective variety $Y$ of dimension $n$. If $L$ is an ample and globelly 
generated line bundle on $Y$ and $m$ is a positive integer, then 
$$H^i(Y, f_*\omega_X^{\otimes m} \otimes L^{\otimes l}) = 0$$
for all $i > 0$ and $l \geq m(n+1)-n$. 
\end{thm}

However, their proof of \cite[Theorem 1.7]{PoSc} does not work 
for the higher direct images $R^jf_*\omega_X^{\otimes m}$. 
They posed the following question in \cite[p. 2293]{PoSc}: 

\begin{ques}\label{ques1.3}
Let $X$ be a smooth projective variety and $f: X \to A$ be a morphism 
to an abelian variety. 
Then is $R^jf_*(\omega_X^{\otimes m})$ a GV-sheaf 
for any $j \geq 1$ and $m \geq 2$?
\end{ques}

We will give an answer of the question by constructing an example that 
$R^jf_*(\omega_X^{\otimes m})$ is not a GV-sheaf 
for some $j \geq 1$ and $m \geq 2$ 
(see Example \ref{ex4.5}).  
The existence of such an example implies that, 
in the notations of Theorem \ref{thm_PoSc2}, we can not obtain any effective 
lower bounds for positive integers $l$ which satisfy that 
$$H^i(Y, R^jf_*\omega_X^{\otimes m} \otimes L^{\otimes l}) = 0$$ 
for any $i >0$ 
(for details, see Remark \ref{rem4.7}).

\begin{ack}
I would like to express my gratitude to 
Professor Osamu Fujino for his much useful advice 
and generous support. 
I would also like to thank my colleagues for discussions 
and the referee for many valuable comments.   
\end{ack}

\section{Preliminaries}\label{sec2}
In this section, we collect some basic definitions.

\begin{say}[Cohomology support loci]\label{say2.1}
Let $X$ be a smooth projective variety, 
$f:X \to A$ be a morphism to an abelian variety 
and $\mathcal F$ be a coherent sheaf on $X$. 
Set $S^i_j(\mathcal F, f) = \{ \xi \in \Pic ^0(A) \mid
h^i(X,\mathcal F \otimes f^*\xi ) \geq j \}$. 
We call $S^i_j(\mathcal F, f)$ 
a \textit{cohomology support locus} of $\mathcal F$. 
This is a Zariski closed subset of the abelian variety $\Pic ^0(A)$. 
If $f$ is the Albanese morphism, we denote $S^i_j(\mathcal F, f)$ by 
$S^i_j(\mathcal F)$. 
We simply denote $S^i_1(\mathcal F,f)$ by $S^i(\mathcal F,f)$. 
\end{say}

\begin{say}[Fourier--Mukai transforms]\label{say2.2}
Let $A$ be an abelian variety
and $\hat{A}$ be the dual abelian variety of $A$. 
We define a functor $\Phi_{\mathcal P_A} : \Coh (A) \to \Coh (\hat{A})$
as 
$\Phi_{\mathcal P_A}(\mathcal F) = p_{2*}(p_1^* \mathcal F \otimes \mathcal P_A)$, 
where $\mathcal P_A$ is the (normalized) Poincar\'e bundle on $A \times \hat{A}$ 
and $p_1:A \times \hat{A} \to A$, $p_2:A \times \hat{A} \to \hat{A}$ are 
the natural projections. 
The right derived functor $R\Phi_{\mathcal P_A}$ of $\Phi_{\mathcal P_A}$
 is called the \textit{Fourier--Mukai transform}. 

Let $X$ be a smooth projective variety. 
Given a morphism $f:X \to A$, 
we put $\mathcal P_X = (f \times \id _{\hat{A}})^*\mathcal P_A$. 
We define a functor $\Phi_{\mathcal P_X}: \Coh (X) \to \Coh(\hat{A})$ 
as $\Phi_{\mathcal P_X}(\mathcal F) = 
\pi_{2*}(\pi_1^*\mathcal F \otimes \mathcal P_X)$, 
where $\pi_1:X \times \hat{A} \to X$ 
and $\pi_2:X \times \hat{A} \to \hat{A}$ are the natural projections. 
\end{say}

\begin{say}[GV-sheaves]\label{say2.3}
Let $X$ be a smooth projective variety, 
$f:X \to A$ be a morphism to an abelian variety 
and $\mathcal F$ be a coherent sheaf on $X$. 
We say that $\mathcal F$ is a 
\textit{GV-sheaf with respect to the morphism $f$} if 
$\codim _{\Pic ^0(A)}S^i(\mathcal F, f) \geq i$ 
for every $i>0$. 
This condition is equivalent to the condition that  
$\codim _{\Pic ^0(A)}\Supp R^i \Phi_{\mathcal P_X}(\mathcal F) \geq i$
for every $i>0$. 
If $f = \id _A$, then we simply call such a sheaf a GV-sheaf on $A$. 
For details, see \cite{PaPo}. 
\end{say}

\begin{say}[$\mathbb Q$-divisors]\label{say2.4}
Let $X$ be a normal variety. 
A \textit{$\mathbb Q$-divisor} on $X$ is a formal sum of 
finitely many prime divisors on $X$ 
with rational coefficients. 
\end{say}

\begin{say}[Operations on $\mathbb Q$-divisors]\label{say2.5}
Let $D = \sum d_jD_j$ be a $\mathbb Q$-divisor on a normal variety $X$, 
where $d_j$ are rational numbers and $D_j$ are prime divisors on $X$. 
We define the \textit{round-down} $\lfloor D \rfloor$ of $D$ 
as $\lfloor D \rfloor = \sum \lfloor d_j \rfloor D_j$, 
where $\lfloor d_j \rfloor$ is the largest integer 
among all the integers that are not larger than $d_j$. 
The \textit{fractional part} $\{ D \}$ of $D$ is defined as 
$\{ D \} = D - \lfloor D \rfloor$. 
Finally, we define 
$$D^+=\sum_{d_j>0}d_jD_j,\ D^-=-\sum_{d_j<0}d_jD_j,\ D^{=1} = \sum_{d_j=1}D_j.$$
Then we have $D = D^+ - D^-$. 
\end{say}

\begin{say}[Unitary local systems]\label{say2.6}
Let $X$ be a complex manifold. 
A \textit{local system} (\textit{of rank one}) on $X$ is a locally constant sheaf 
of one-dimensional $\mathbb C$-vector spaces on $X$.
Then there is a natural correspondence between 
local systems on $X$ and characters of the fundamental group $\pi_1(X)$  
(a \textit{character} of $\pi_1(X)$ is a group homomorphism 
$\rho : \pi_1(X) \to \mathbb C^*$). 
A \textit{unitary local system} on $X$ is a local system 
that corresponds to a character whose image is contained in $U(1)$. 
We denote the set of all unitary local systems on $X$ by $U_B(X)$. 

Let $\Pic ^\tau(X)$ be the subgroup of $\Pic (X)$ 
consisting of line bundles whose first Chern classes are torsion elements. 
Take $L \in \Pic ^\tau(X)$. 
$L$ has the Chern connection $\nabla : L \to \Omega^1_X \otimes L$ 
and $\Ker \nabla$ is a unitary local system on $X$. 
The correspondence $L \mapsto \Ker \nabla$ gives an isomorphism 
$\Pic ^\tau(X) \cong U_B(X)$.
For details, see \cite[Chapter 1, Section 2 and Section 4]{Kob}. 
\end{say}

\section{Cohomology support loci of pluricanonical bundles}\label{sec3}

First, we refer to 
Timmerscheidt's mixed Hodge theory for unitary local systems (see \cite{Tim}). 
Let $X$ be a smooth projective variety, 
$\Delta$ be a simple normal crossing divisor on $X$ 
and $U=X \setminus \Delta$. 
Take $\mathcal V \in U_B(U)$. 
Then we have Deligne's canonical extension $(L, \nabla)$ of $\mathcal V$, 
where $L$ is a holomorphic line bundle  on $X$ 
and $\nabla$ is a logarithmic connection  on $L$ 
with poles along $\Delta$ 
such that $\Ker \nabla |_U = \mathcal V$ (for details, see \cite[(1.3)]{EsVi} 
and \cite[p. 226]{Bud}). 
We have the logarithmic de Rham complex 
$$ 0 \to L \stackrel{\nabla}{\to} \Omega^1_X(\log \Delta) \otimes L  
\stackrel{\nabla}{\to} \Omega^2_X(\log \Delta) \otimes L \stackrel{\nabla}{\to} \cdots .$$
We give to the complex $(\Omega^{\bullet}_X(\log \Delta) \otimes L, \nabla)$ 
the following filtration $F$: 
$$ F^p(\Omega^q_X(\log \Delta) \otimes L) = 
\begin{cases}
0 & \text{for $q < p$,} \\
\Omega^p_X(\log \Delta) \otimes L & \text{for $q \geq p$.}
\end{cases}
$$
Let $\iota: U \to X$ be the inclusion. 
$R\iota_*\mathcal V$ is quasi-isomorphic to 
$\Omega^{\bullet}_X(\log \Delta) \otimes L$ (see \cite[(1.4)]{EsVi}). 
Therefore the filtration $F$  induces a spectral sequence 
$$ E^{p, q}_1 = H^q(X,\Omega^p_X(\log \Delta) \otimes L)
\Rightarrow \mathbb H ^{p+q}(X, \Omega^{\bullet}_X(\log \Delta) \otimes L)
= H^{p+q}(U,\mathcal V).$$

In general, we can not expect a natural mixed Hodge structure on 
$H^{p+q}(U, \mathcal V)$ when $\mathcal V$ is not defined over $\mathbb R$. 
However, Timmerscheidt proved: 

\begin{thm}[{\cite{Tim}}, Theorem (7.1)]\label{thm3.1}
In the above notations,
 the spectral sequence degenerates at $E_1$. 
\end{thm}

\begin{rem}\label{rem3.1.1}
Theorem \ref{thm3.1} holds for unitary local systems of arbitrary rank. 
\end{rem}

Next, we refer to a result of Budur. 
The following statement is a special case 
of \cite[Theorem 8.3]{Bud}: 

\begin{thm}[{\cite{Bud}}]\label{thm3.2}
Let $X$ be a smooth projective variety, 
$\Delta$ be a simple normal crossing divisor on $X$ 
and $U=X \setminus \Delta$. 
Then the image of the set 
$$ \{ \mathcal V \in U_B(X) \mid \dim \Gr ^p_F H^{p+q}(U, \mathcal V|_U) \geq j \} $$
in $\Pic ^\tau (X)$ is a finite union of 
torsion translates of abelian subvarieties for every $p$, $q$ and $j$. 
\end{thm}

Take a unitary local system $\mathcal V \in U_B(X)$ 
on $X$ 
and let $L \in \Pic ^\tau(X)$ be the corresponding line bundle. 
Then the canonical extension of $\mathcal V |_U$ is $L$ 
with the Chern connection. 
So by Theorem \ref{thm3.1},  
we can restate Theorem \ref{thm3.2} as follows:

\begin{thm}\label{thm3.3}
Let $X$ be a smooth projective variety, 
$\Delta$ be a simple normal crossing divisor on $X$ 
and $U=X \setminus \Delta$. 
Then the set 
$$ \{ L \in \Pic ^\tau (X) \mid \dim H^q(X,\Omega^p_X(\log \Delta) \otimes L) \geq j \} $$
is a finite union of torsion translates of abelian subvarieties for every $p$, $q$ and $j$. 
\end{thm}

By using Theorem \ref{thm3.3}, we can prove Theorem \ref{thm1.1} in Section 
\ref{sec1}: 

\begin{thm}[= Theorem {\ref{thm1.1}}]\label{thm3.4}
Let $X$ be a smooth projective variety, 
$\Delta$ be a boundary $\mathbb Q$-divisor on $X$ 
with simple normal crossing support, 
$f:X \to A$ be a morphism to an abelian variety 
and $D$ be a Cartier divisor on $X$ 
such that $D \sim _{\mathbb Q} K_X + \Delta$. 
Then $S^i_j(D, f)$ is a finite union of torsion translates of 
abelian subvarieties of $\Pic ^0(A)$ 
for every $i \geq 0$ and $j \geq 1$.
\end{thm}

\begin{proof}
We divide the proof into two steps.

\begin{step}
We reduce to the case when 
$\Delta$ is a (\textit{reduced}) simple normal crossing divisor. 

Set $C = D-K_X-\lfloor \Delta \rfloor$.
Then $C \sim_{\mathbb Q} \{ \Delta \}$, 
so we can take a positive integer $N$ 
such that $NC \sim N\{ \Delta \}$. 
Let $\pi : Y \to X$ be the normalization of the cyclic cover 
$$ \Spec \bigoplus _{k=0}^{N-1} \mathcal O_X(-kC). $$ 
Then it follows that 
$$ \pi_* \mathcal O_Y 
= \bigoplus_{k=0}^{N-1}\mathcal O_X(-kC+ \lfloor k \{ \Delta \} \rfloor) $$
(see \cite[Proposition 9.8]{Kol}). 
So 
$$ \pi_* \mathcal O_Y(-\pi^* \lfloor \Delta \rfloor)
= \bigoplus_{k=0}^{N-1} \mathcal O_X(-kC+ \lfloor k \{ \Delta \} \rfloor 
-\lfloor \Delta \rfloor ) $$
contains $\mathcal O_X(-C-\lfloor \Delta \rfloor) = \mathcal O_X(-(D-K_X))$ 
as a direct summand. 
Then by the standard argument (cf. the proof of \cite[Theorem 6.1]{Bud}) and 
Serre duality, it follows that 
$S^i_j(D,f)$ is a finite union of torsion translates of 
abelian subvarieties for every $i$ and $j$ 
if $S^i_j(-\pi^*\lfloor \Delta \rfloor,f \circ \pi)$ is a finite union of torsion translates of abelian subvarieties 
for every $i$ and $j$. 

So we consider $(Y, \pi^*\lfloor \Delta \rfloor)$. 
Note that $\pi^*\lfloor \Delta \rfloor$ is a reduced divisor. 
Take a divisor $\Delta_Y$ on $Y$ as $K_Y + \Delta_Y = \pi^*(K_X + 
\Delta_{\red})$. 
Then, by Hurwitz's formula, we see that $\Delta_Y$ is reduced and 
$\pi^*\lfloor \Delta \rfloor \leq \Delta_Y$. 
Moreover, $(Y, \Delta_Y)$ is a log canonical pair since $(X, \Delta_{\red})$ is 
log canonical and $\pi$ is finite (see \cite[Proposition 5.20]{KoMo}). 
These arguments imply that $(Y,  \pi^*\lfloor \Delta \rfloor)$ is a log canonical 
pair. 

Take a log resolution $\mu: Y' \to Y$ of $(Y, \pi^*\lfloor \Delta \rfloor)$. 
Set 
$$ \Delta_{Y'}= \mu^*(K_Y+\pi^* \lfloor \Delta \rfloor) - K_{Y'}, $$
then 
$$ \Delta_{Y'}^{=1}= \mu^{-1}_*(\pi^* \lfloor \Delta \rfloor)+E $$
for some reduced $\mu$-exceptional divisor $E$. 
Since $-K_{Y'/Y} = -(K_{Y'} - \mu^*K_Y)$ has no irreducible components 
with coefficient 1, 
every component of $E$ is contained in 
$\mu^*\pi^*\lfloor \Delta \rfloor$. 
Since $E$ is $\mu$-exceptional, $E$ is in fact contained in 
$\mu^*\pi^*\lfloor \Delta \rfloor - \mu_*^{-1}\pi^*\lfloor \Delta \rfloor $. 
So $F= \mu^*\pi^*\lfloor \Delta \rfloor - 
\mu_*^{-1}\pi^*\lfloor \Delta \rfloor -E$ 
is an effective and $\mu$-exceptional divisor on $Y'$. 
By the Fujino--Kov\'acs vanishing theorem (see \cite{Kov} and \cite{Fuj2}), 
$$ R^i\mu_* \mathcal O_{Y'}(-\Delta_{Y'}^{=1}) = 0 $$
for $i>0$. 
Therefore    
\begin{align*}
R\mu_* \mathcal O_{Y'}(-\Delta_{Y'}^{=1}) 
&\cong \mu_* \mathcal O_{Y'}(-\Delta_{Y'}^{=1}) \\
&\cong \mu_* \mathcal O_{Y'}(-\mu_*^{-1}\pi^* \lfloor \Delta \rfloor -E) \\
&\cong \mu_* \mathcal O_{Y'}(-\mu^*\pi^* \lfloor \Delta \rfloor +F) \\
&\cong \mathcal O_Y(-\pi^* \lfloor \Delta \rfloor).
\end{align*} 
Thus we have $S^i_j(-\Delta_{Y'}^{=1}, f \circ \pi \circ \mu)
= S^i_j(-\pi^* \lfloor \Delta \rfloor, f \circ \pi)$. 
Moreover, $\Delta_{Y'}^{=1}$ is a simple normal crossing divisor on $Y'$. 
So it is sufficient to give a proof in the case when 
$\Delta$ is a simple normal crossing divisor. 
\end{step}

\begin{step}
Assume that $\Delta$ is a simple normal crossing divisor. 
Put $U=X \setminus \Delta$. 
By Theorem \ref{thm3.3}, the set 
\[ \{ L \in \Pic^\tau (X) \mid h^q(X, \Omega^p_X(\log \Delta) \otimes L) \geq j \} \] 
is a finite union of torsion translates of 
abelian subvarieties for any $p$, $q$ and $j$. 
In particular, $S^i_j(K_X+\Delta)=S^i_j(K_X+\Delta, \alb _X)$ is a finite union of 
torsion translates of abelian subvarieties for any $i$ and $j$, 
where $\alb_X$ denotes the Albanese morphism of $X$. 
Considering the universality of the Albanese morphism, 
it follows from this that 
$S^i_j(K_X+\Delta, f)$ is also a finite union of 
torsion translates of abelian subvarieties. 
\end{step}
\end{proof}

\begin{rem}\label{rem_for_thm3.4}
Note that the pair $(X, \Delta)$ in the statement of Theorem \ref{thm3.4}  
is not an arbitrary log canonical pair. 
The technical reason why Theorem \ref{thm3.4} 
for arbitrary log canonical pairs is not made 
is that the arguments for unitary local systems require 
log smoothness of $(X, \Delta)$
(i.e. the condition that 
$X$ is smooth and $\Delta$ has simple normal crossing support), and 
the statement for arbitrary log canonical pairs cannot easily be reduced 
to the log smooth case. 
On the other hand, Theorem \ref{thm3.7} below holds for arbitrary log canonical 
pairs.  
\end{rem}

We can also prove a generalization of the generic vanishing theorem for log 
canonical pairs.

\begin{thm}\label{thm3.5}
Let $X$ be a smooth projective variety, 
$\Delta$ be a boundary $\mathbb Q$-divisor on $X$ 
with simple normal crossing support, 
$f:X \to A$ be a morphism to an abelian variety 
and $D$ be a Cartier divisor on $X$ 
such that $D \sim _{\mathbb Q} K_X + \Delta$. 
Set 
$l = \max\{ \dim V - \dim f(V) \mid V=X\ \mathrm{or}\ V\ \mathrm{is\ a\ 
log\ canonical\ center\ of}\ (X, \Delta) \}$.
Then 
$$\codim _{\Pic^0(A)} S^i(D, f) \geq i - l$$
for any $i$. 
\end{thm}

For the definition of log canonical centers, see \cite[Definition 4.6]{Fuj1}. 

\begin{proof}
Let $S=\lfloor \Delta \rfloor$ and $\Delta_i$ be an irreducible component of $S$.
Consider the exact sequence 
$\cdots \to R^jf_*\mathcal O_X(D-\Delta_i) \to R^jf_*\mathcal O_X(D) 
\to R^jf_*\mathcal O_{\Delta_i}(D|_{\Delta_i}) \to \cdots$.
Then it follows that $R^jf_*\mathcal O_X(D)=0$ for $j >l$ by induction of both
the dimension of $X$ and the number of irreducible components of $S$.

Consider the Leray spectral sequence
$$E^{p,q}_2(\xi) = H^p(A, R^qf_*\mathcal O_X(D) \otimes \xi) \Rightarrow 
H^{p+q}(X, \mathcal O_X(D) \otimes f^*\xi),$$
where $\xi \in \Pic ^0(A)$. 
Then it follows by the spectral sequence that 
$$ S^i(D, f) \subset \bigcup_{q=0}^l S^{i-q}(R^qf_*\mathcal O_X(D)).$$
Furthermore, $R^qf_*\mathcal O_X(D)$ are GV-sheaves on $A$ for all $q$ (see 
Proposition \ref{prop4.2} below), so 
$$\codim S^{i-q}(R^qf_*\mathcal O_X(D)) \geq i-q$$ 
for $0 \leq i \leq l$. 
Hence $\codim S^i(D, f) \geq i-l$. 
\end{proof}

\begin{rem}\label{rem3.6}
In the notations of Theorem \ref{thm3.5}, assume in addition that 
$(X, \Delta)$ is a 
Kawamata log terminal pair. Then $l = \dim X - \dim f(X)$. 
In particular, assume $\Delta=0$, then Theorem \ref{thm3.5} is the original 
generic vanishing theorem of Green and Lazarsfeld in \cite{GrLa}. 
\end{rem}

Next, we consider the cohomology support loci of $m(K_X+\Delta)$. 

\begin{thm}[= Theorem {\ref{thm1.2}}]\label{thm3.7}
Let $X$ be a normal projective variety, 
$\Delta$ be a boundary $\mathbb Q$-divisor on $X$ 
such that $K_X+\Delta$ is $\mathbb Q$-Cartier and $(X, \Delta)$ is 
log canonical, 
$f:X \to A$ be a morphism to an abelian variety, 
and $D$ be a Cartier divisor on $X$ 
such that $D \sim _{\mathbb Q} m(K_X + \Delta)$ for some positive integer $m$. 
Then $S^0_j(D, f)$ is a finite union of 
torsion translates of abelian subvarieties of $\Pic ^0(A)$ 
for any $j \geq 1$.
\end{thm}

\begin{proof}
Fix an element $\xi \in S^0_j(D, f)$. It is sufficient to show that 
there is a torsion translate of an abelian subvariety $T \subset \Pic ^0(A)$ 
such that $\xi \in T \subset S^0_j(D, f)$. 
This is because, once we show that 
$S^0_j(D, f)$ is a union of (possibly infinitely many) torsion translates of 
abelian subvarieties of $\Pic^0(A)$, 
then it follows that $S^0_j(D, f)$ is a union of finitely many of them 
since $S^0_j(D, f)$ and all torsion translates of abelian subvarieties are 
Zariski closed subsets of $\Pic^0(A)$, and there are only countably many of 
torsion translates of abelian subvarieties. 

\setcounter{step}{0}
\begin{step}
Take a projective birational morphism $\mu : X' \to X $
satisfying the following conditions: 

(i) $X'$ is a smooth projective variety; 

(ii) $K_{X'} + \Delta' = \mu(K_X + \Delta) + E$, 
where $\Delta'$ is a boundary $\mathbb Q$-divisor on $X'$ 
and $E$ is an effective and $\mu$-exceptional Cartier divisor on $X'$; 

(iii) $\mu^* |D + f^*\xi|= |V| + F$, 
where $|V|$ is a free linear system and 
$F$ is an effective Cartier divisor on $X'$; 

(iv) $\Exc (\mu) \cup \Supp \Delta' \cup \Supp F$ 
is a simple normal crossing divisor on $X'$. 

$D' = \mu^*D+mE$ is a Cartier divisor on $X'$ 
satisfying that 
$$ D' \sim_{\mathbb Q} m\mu^*(K_X+\Delta)+mE
=m(K_{X'}+\Delta').$$ 
For any $\eta \in \Pic ^0(A)$, 
$h^0(X', \mathcal O_{X'}(D'+\mu^*f^*\eta)) = 
h^0(X, \mathcal O_X(D+f^*\eta))$.
Therefore $S^0_j(D', f \circ \mu)=S^0_j(D,f)$.
Furthermore, 
\begin{align*}
|D'+\mu^*f^*\xi| &= |\mu^*(D+f^*\xi)+mE|
\\ &= \mu^*|D+f^*\xi| + mE 
\\ &= |V|+F+mE.
\end{align*}
By replacing $(X, \Delta, D)$ with $(X', \Delta', D')$, 
we can assume that the following conditions hold: 

$X$ is smooth,  
$|D+f^*\xi| = |V|+F$ for some free linear system $|V|$ and some 
effective Cartier divisor $F$, and $\Supp \Delta \cup \Supp F$ is 
a simple normal crossing divisor. 
\end{step}

\begin{step}
Since $\xi \in S^0_j(D, f)$, $\dim |V| \geq j-1$. 
Take a general smooth divisor $M \in |V|$ 
with no irreducible component 
contained in $\Supp \Delta \cup \Supp F$. 
We can take $\xi_0 \in \Pic ^0(A)$ such that $\xi = m\xi_0$ in $\Pic ^0(A)$.  
Then 
\begin{align*}
m(K_X+\Delta+f^*\xi_0) 
&= m(K_X+\Delta)+f^*\xi \\ 
&\sim_{\mathbb Q} D+f^*\xi \\
&\sim M+F. 
\end{align*}
Therefore 
$$ K_X+\Delta+f^*\xi_0 \sim_{\mathbb Q} \frac{1}{m}(M+F). $$ 
Then we have 
\begin{align*}
D-F+(m-1)f^*\xi_0 &\sim_{\mathbb Q} m(K_X+\Delta)-F+(m-1)f^*\xi_0
\\ &= K_X+\Delta-F+(m-1)(K_X+\Delta+f^*\xi_0)
\\ &\sim_{\mathbb Q} K_X+\Delta-\frac{1}{m}F+\frac{m-1}{m}M.
\end{align*}
Set 
$$ G=\left\lceil \left(\Delta-\frac{1}{m}F\right)^- \right\rceil \geq 0$$
and
$$ D_0=D-F+(m-1)f^*\xi_0+G.$$
Then 
$$ D_0 \sim_{\mathbb Q} 
K_X+\Delta-\frac{1}{m}F +\frac{m-1}{m}M + G, $$
and 
$$ \Delta-\frac{1}{m}F +\frac{m-1}{m}M + G$$
is a boundary $\mathbb Q$-divisor 
with simple normal crossing support. 
By Theorem \ref{thm3.4}, $S^0_j(D_0, f)$ is a finite union of 
torsion translates of abelian subvarieties. 
Moreover,  
\begin{align*}
h^0(X, \mathcal O_X(D_0+f^*\xi_0)) &= h^0(X, \mathcal O_X(D-F+f^*\xi+G))
\\ &\geq h^0(X,\mathcal O_X(D-F+f^*\xi))
\\ &= \dim |V|+1 = j.
\end{align*}
Therefore we have $\xi_0 \in S^0_j(D_0,f)$. 
\end{step}

\begin{step}
Here we will prove that 
$$S^0_j(D_0,f)+(m-1)\xi_0 \subset S^0_j(D,f).$$

Take $\alpha \in S^0_j(D_0,f)$. 
Then $D+f^*(\alpha+(m-1)\xi_0)=D_0+F-G+f^*\alpha$. 
$F-G$ is obviously effective since $\Delta$ is effective. 
So we have 
$$h^0(X, \mathcal O_X(D_0 + F - G + f^*\alpha)) 
\geq h^0(X, \mathcal O_X(D_0+f^*\alpha)) \geq j.$$ 
Thus $\alpha + (m-1)\xi_0 \in S^0_j(D,f)$, 
so it follows that $$S^0_j(D_0,f)+(m-1)\xi_0 \subset S^0_j(D,f).$$
\end{step}

\begin{step}
By Step 2, we can take 
an abelian subvariety $B \subset \Pic ^0(A)$ 
and a torsion point $q \in \Pic ^0(A)$ 
such that $\xi_0 \in B+q \subset S^0_j(D_0, f)$. 
By Step 3, $B+q+(m-1)\xi_0 \subset S^0_j(D, f)$. 
Represent $\xi_0 \in B+q$ as $\xi_0 = t+q$ 
for some $t \in B$, then 
$B+q+(m-1)\xi_0 = B+mq+(m-1)t=B+mq$. 
Therefore $\xi = \xi_0 + (m-1)\xi_0 
\in B+q+(m-1)\xi_0=B+mq \subset S^0_j(D,f)$.
\end{step}
\end{proof}

\begin{proof}[Proof of Corollary {\ref{thm_CKP}}]
By Theorem \ref{thm3.7}, 
we may assume that $\rho$ is a torsion element.
We take a positive integer $N$ such that $N \cdot \rho = 0$. 
Then $h^0(Nm(K_X+\Delta)) = h^0(N(m(K_X+\Delta)+\rho)) 
\geq h^0(m(K_X+\Delta)+\rho)$. 
\end{proof}

As another application of Theorem \ref{thm3.7}, 
we give an alternative proof of the following proposition in \cite{CKP}:

\begin{prop}[{\cite[Corollary 3.2]{CKP}}]\label{prop3.8}
Let $X$ be a normal projective variety, 
$\Delta$ be a boundary $\mathbb Q$-divisor on $X$ 
such that $K_X+\Delta$ is $\mathbb Q$-Cartier 
and $(X, \Delta)$ is log canonical, 
and $f:X \to A$ be a morphism to an abelian variety.
Then $\kappa(K_X+\Delta) \geq \kappa(K_X+\Delta +f^*\xi)$ 
for every $\xi \in \Pic^0(A)$. 
\end{prop}

In \cite{CKP}, this is proved as a corollary of \cite[Theorem 0.1]{CKP} (see 
Corollary \ref{thm_CKP}). 

First we prove the following lemma:

\begin{lem}\label{lem3.9}
Let $A$ be an abelian variety and $\mathcal T$ be the set of all torsion translates 
of abelian subvarieties of $A$. For $T \in \mathcal T$, 
set $\ord(T) = \min \{  \ord(t) \mid t \in T_{\tor} \}$, where 
$T_{\tor}$ denotes the set of all torsion points in $T$. 

{\rm (i)} $\ord(mT) \leq \ord(T) \leq m \cdot \ord(mT)$ for any $T \in \mathcal T$ and 
$m \in \mathbb Z_{>0}$. 

{\rm (ii)} For any $p \in A$, the set 
$\{ \ord(T) \mid T \in \mathcal T\ \mathrm{and}\ mp \in T\ 
\mathrm{for\ some\ }m \in 
\mathbb Z_{>0} \}$ is bounded. 
\end{lem}

\begin{proof}
(i) By definition, $\ord(t_0)=\ord(T)$ for some $t_0 \in T$. Then 
$$\ord(mT) \leq \ord(mt_0) \leq \ord(t_0) = \ord(T).$$

Take $t \in T_{\tor}$. 
Then $\ord(T) \leq \ord(t) \leq m\ord(mt)$ since 
$$m\ord(mt)\cdot t = 
\ord(mt) \cdot mt=0.$$
So $\ord(T) \leq m\ord(mT)$. 

(ii) Suppose that 
$\{ \ord(T) \mid T \in \mathcal T\ \mathrm{and} 
\ mp \in T\ \mathrm{for\ some\ }m \in 
\mathbb Z_{>0} \}$ is unbounded. 
Consider a sequence
$S_k = (m_1,\cdots,m_k; T_1,\cdots,T_k)$ satisfying that 
\begin{itemize}
\item $m_1,\cdots,m_k \in \mathbb Z_{>0}$, $T_1,\cdots,T_k \in \mathcal T$, 
\item $m_1p \in T_1,\ m_1m_2p \in T_2,\cdots,m_1\cdots m_kp \in T_k$, 
\item $\dim T_1 > \dim T_2 > \cdots > \dim T_k$. 
\end{itemize}
By assumption, we can take $m_{k+1} \in \mathbb Z_{>0}$ and $T''_{k+1} \in 
\mathcal T$ such that $m_{k+1}p \in T''_{k+1}$ and $\ord(T''_{k+1}) > 
m_1 \cdots m_k \ord(T_k)$. 
We define $T'_{k+1}, T'_k \in \mathcal T$ as $T'_{k+1} = m_1 \cdots m_k 
T''_{k+1}$ and $T'_k = m_{k+1}T_k$. 
Then we have 
$$\ord(T'_{k+1}) \geq \frac{\ord(T''_{k+1})}{m_1 \cdots m_k} > \frac{m_1 \cdots m_k 
\ord(T_k)}{m_1\cdots m_k} = \ord(T_k) \geq \ord(T'_k),$$
using (i). 
Put $T_{k+1} = T'_{k+1} \cap T'_k \in \mathcal T$. 
Then $m_1 \cdots m_{k+1}p \in T_{k+1}$. 
Since $\ord(T'_{k+1}) > \ord(T'_k)$, $T'_k \not\subset T'_{k+1}$. 
Then it follows that $T_{k+1} \subsetneq T'_k$. 
Therefore $\dim T_{k+1} < \dim T'_k = \dim T_k$. 

By continuing this process, we have $S_1, S_2, \cdots, S_k, S_{k+1}, \cdots$. 
But this process must stop because of the third condition of $S_k$.  
So we have a contradiction. 
\end{proof}

\begin{proof}[Proof of Proposition \ref{prop3.8}]
Let $\mathcal T$ be the set of all torsion translates of abelian subvarieties of 
$\Pic^0(A)$. 
There exists a positive integer $m_0$ such that $D = m_0(K_X + \Delta)$ is 
Cartier. 
Then it reduces to show that $\kappa(D) \geq \kappa(D+f^*\xi)$ for any $\xi \in 
\Pic^0(A)$. 
Set $d(m) = h^0(mD)$ and $d'(m) = h^0(m(D+f^*\xi))$. 
Then $m\xi \in S^0_{d'(m)}(mD, f)$, so by Theorem \ref{thm3.7}  
we can take $T_m \in \mathcal T$ 
such that $m\xi \in T_m \subset S^0_{d'(m)}(mD, f)$. 
Take $t_m \in T_m$ such that $\ord(t_m) = \ord(T_m)$. 
It follows by Lemma \ref{lem3.9} (ii) that 
there exists a positive integer $N$ satisfying that $\ord(T_m) \mid N$ for any 
$m \in \mathbb Z_{>0}$. 
Hence $N t_m=0$ for any $m \in \mathbb Z_{>0}$. 
Then we have 
\begin{align*}
d'(m)   &= h^0(m(D+f^*\xi)) \\
          &\leq h^0(mD+f^*t_m) \\
          &\leq h^0(N(mD+f^*t_m)) \\
          &= h^0(NmD) = d(Nm). 
\end{align*}
Therefore 
\begin{align*}
\kappa(D+f^*\xi) &= \limsup_m \frac{\log d'(m)}{\log m} \\
                          &\leq \limsup_m \frac{\log d(Nm)}{\log m} \\
                          &= \limsup_m \frac{\log d(Nm)}{\log Nm} = \kappa(D). 
\end{align*}
\end{proof}

\section{Higher direct images of pluricanonical bundles}\label{sec4}

In this section, we discuss higher direct images of pluricanonical bundles 
to abelian varieties. 

In \cite{PoSc}, 
Popa and Schnell proved the following theorem: 

\begin{thm}[{\cite[Theorem 1.10]{PoSc}}]\label{thm4.1}
Let $X$ be a smooth projective variety 
and $f:X \to A$ be a morphism to an abelian variety.
Then the direct image $f_*(\omega_X^{\otimes m})$ 
is a GV-sheaf for any positive integer $m$. 
\end{thm}

In addition, we can prove a log version of Theorem \ref{thm_Hac}:

\begin{prop}\label{prop4.2}
Let $X$ be a smooth projective variety, 
$\Delta$ be a boundary $\mathbb Q$-divisor on $X$ 
with simple normal crossing support, 
$f:X \to A$ be a morphism to an abelian variety 
and $D$ be a Cartier divisor on $X$ 
such that $D \sim _{\mathbb Q} K_X + \Delta$. 
Then the higher direct images $R^jf_*\omega_X(\Delta)$ are GV-sheaves for any $j$.
\end{prop}

Here we use a characterization of GV-sheaves on abelian varieties by Hacon. 
He proved the following criterion in \cite{Hac}:

\begin{thm}[Hacon]\label{thm4.3}
Let $A$ be an abelian variety and $\mathcal F$ be a coherent sheaf on $A$. 
$\mathcal F$ is a GV-sheaf on $A$ if and only if 
$H^k(B, \phi^*\mathcal F \otimes L)=0$ for $k>0$, 
where $\phi:B \to A$ is any isogeny and $L$ is any ample line bundle on $B$. 
\end{thm}

\begin{proof}[Proof of Propotion \ref{prop4.2}]
Let $\phi: B \to A$ be an isogeny and $L$ be an ample line bundle on $B$. 
Set $Y = X \times_A B$. Let $\psi: Y \to X$ and $g: Y \to B$ be the natural 
morphisms. Fix an integer $j \geq 0$. 
By base change, 
$$\phi^*R^jf_*\omega_X(\Delta) \cong R^jg_*\psi^*\omega_X(\Delta) 
\cong R^jg_*\omega_Y(\psi^*\Delta).$$
Applying the Ambro--Fujino vanishing theorem (see \cite[Theorem 3.2]{Amb} 
and \cite[Theorem 6.3]{Fuj1}) 
to the log canonical pair 
$(Y, \psi^*\Delta)$, we have 
$$H^k(B, \phi^*R^jf_*\omega_X(\Delta) \otimes L) \cong 
H^k(B, R^jg_*\omega_Y(\psi^*\Delta) \otimes L)=0$$
for $k>0$. 
By Theorem \ref{thm4.3}, it follows that $R^jf_*\omega_X(\Delta)$ is a 
GV-sheaf. 
\end{proof}

In \cite{PoSc}, Popa and Schnell also asked the question  
whether the higher direct images $R^jf_*(\omega_X^{\otimes m})$ 
are GV-sheaves or not. 
Note that $R^jf_*\omega_X$ is a GV-sheaf for every $j$ 
(see Theorem \ref{thm_Hac} or Proposition \ref{prop4.2}).
Hence the question is whether $R^jf_*(\omega_X^{\otimes m})$ is 
a GV-sheaf or not in the case when $j \geq 1$ and $m \geq 2$.

Here we give an example of 
a higher direct image of a pluricanonical bundle 
which is not a GV-sheaf. 

\begin{lem}\label{lem4.4}
There exists an irregular smooth projective variety 
with big anti-canonical bundle. 
\end{lem}

\begin{proof}
Let $A$ be an abelian variety. We take an
ample line bundle $L$ on $A$ and define a vector bundle $E$ 
as the direct sum of $L^{-1}$ and $\mathcal O_A$. 
Let $\pi: X = \mathbb P_A(E) \to A$ be the projective bundle on $A$ 
associated to $E$. 
Clearly the irregularity $q(X)$ of $X$ is positive.
The canonical bundle $\omega_X$ is isomorphic to 
$\pi^*(\omega_A \otimes \det E ) \otimes \mathcal O_X(-\rank E) = 
\pi^*L^{-1} \otimes \mathcal O_X(-2)$ 
(see \cite[7.3.A]{Laz}). 

We will see that $\omega_X^{-1}$ is big.
Let $\xi$ and $l$ be the numerical classes of $\mathcal O_X(1)$ 
and $L$, respectively. 
Note that $\xi$ is an effective class since $H^0(X, \mathcal O_X(1) ) = 
H^0(A, E) \neq 0$. 
The numerical class of $\omega_X^{-1}$ is equal to 
$$ 2\xi + \pi^*l = \frac{N-1}{N}2\xi +\frac{1}{N}2\xi + \pi^*l,$$ 
where $N$ is a sufficiently large integer such that 
$(1/N)2\xi + \pi^*l$ is ample. 
So the numerical class of $\omega_X^{-1}$ is represented by 
the sum of an effective class and an ample class. 
Therefore $\omega_X^{-1}$ is big.
\end{proof}

\begin{ex}\label{ex4.5}
Let $X$ be an irregular smooth projective variety of dimension $n \geq 2$ 
with big anti-canonical bundle 
and let $f = \alb _X: X \to A$ be the Albanese morphism of $X$.  
Now we show that $R^jf_*\omega_X^{\otimes m}$ is not a GV-sheaf 
for some positive integers $j$ and $m$. 

Let $\mathcal P _A$ be the Poincar\'e bundle on $A \times \hat{A}$ 
and $\mathcal P _X = (f \times \id _{\hat{A}})^*\mathcal P_A$. 
Then we have the Fourier--Mukai transforms 
$\Phi_{\mathcal P_A} : \Coh (A) \to \Coh (\hat{A})$ and 
$\Phi_{\mathcal P_X} : \Coh (X) \to \Coh (\hat{A})$ 
(see \ref{say2.2}). 
It immediately follows that $\Phi_{\mathcal P_X} \cong \Phi_{\mathcal P_A} \circ f_* $. 

Now we prove the following lemma:

\begin{lem}\label{lem4.6}
Let $X$ be a smooth projective variety of dimension $n$ 
and $D$ be a big Cartier divisor on $X$. 
Then $$S^0(mD) = \{ \xi \in \Pic ^0(X) \mid 
H^0(X, \mathcal O_X(mD + \xi)) \neq 0 \} = \Pic ^0(X)$$
for any sufficiently large and divisible $m$.
\end{lem}

\begin{proof}
Since $D$ is big, there exist a positive integer $m_0$, 
a very ample Cartier divisor $H$, and an effective Cartier divisor $E$ such that 
$m_0D \sim H + E$. 
For any positive integer $m$, we have  
$$S^0(mm_0D) = S^0(mH + mE) \supset S^0(mH).$$
We can take a positive integer $m_1$ satisfying that
$$H^i(X, \mathcal O_X(mH+ \xi)) = 0$$
for every $\xi \in \Pic ^0(X)$, $m \geq m_1$ and $i > 0$
(take $m_1$ such that $m_1H-K_X$ is ample). 
According to the notion of the Castelnuovo--Mumford regularity, 
$mH + \xi$ is 0-regular for every $\xi \in \Pic ^0(X)$ and $m \geq m_1 + n$, 
and so it is globally generated. 
In particular, $S^0(mH) = \Pic ^0(X)$ for every $m \geq m_1 + n$. 
Therefore $S^0(mm_0D) = \Pic ^0(X)$ for every $m \geq m_1 + n$.
\end{proof}

By the above lemma, we can take a positive integer $m$ such that 
$S^n(\omega_X ^{\otimes m}) =-S^0(\omega_X ^{\otimes (1-m)} ) = \Pic ^0(X)$. 
Then it follows that 
$$ \Supp R^n\Phi_{\mathcal P _X}(\omega_X^{\otimes m})= \Pic ^0(X).$$
Consider the Grothendieck spectral sequence 
\[ E^{i, j}_2 = R^i\Phi_{\mathcal P_A} R^jf_*(\omega_X^{\otimes m}) 
\Rightarrow R^{i + j}\Phi_{\mathcal P_X}(\omega_X^{\otimes m}).
\]
Then there exists an integer $i$ such that 
$\Supp R^i\Phi_{\mathcal P_A} R^{n-i}f_*(\omega_X^{\otimes m}) = \Pic ^0(X)$. 
Note that $\dim f(X) > 0$ since $X$ is irregular. 
Therefore $R^nf_*(\omega_X^{\otimes m}) = 0$, and so $i$ must be positive. 
Note that a coherent sheaf $\mathcal F$ on $A$ is a GV-sheaf 
if and only if $\codim \Supp R^j\Phi_{\mathcal P_A} \mathcal F 
\geq j $ for $j > 0$ (see \ref{say2.3}).  
Hence it follows that 
$R^{n-i}f_*(\omega_X^{\otimes m})$ is not a GV-sheaf. 
\end{ex}

\begin{rem}\label{rem4.7}
Let $X$ be a smooth projective variety of dimension $k$, 
$f: X \to Y$ be a morphism to a 
projective variety $Y$ of dimension $n$, $L$ be an ample and globally generated 
line bundle on $Y$ and $m > 0$, $j \geq 0$ be integers.  
Suppose that we can take a positive integer $N = N(k, n, m, j)$ depending only on 
$k, n, m$ and $j$ such that 
$$H^i(Y, R^jf_*\omega_X^{\otimes m} \otimes L^{\otimes l}) = 0$$ 
for any $i>0$ and $l \geq N$. 
Then, by the same argument as the proof of Theorem \ref{thm_PoSc} by using 
Theorem \ref{thm_PoSc2} (see the proof of \cite[Theorem 1.10]{PoSc}), 
it follows that $R^jf_*\omega_X^{\otimes m}$ is a GV-sheaf for any morphism 
$f: X \to A$ to an abelian variety, $j \geq 0$ and $m \geq 1$. 
But this contradicts Example \ref{ex4.5}. 
So we can not take such an integer $N$. 
\end{rem}

Next, we consider the case when $X$ is a surface. 
Example \ref{ex4.5} shows that 
$R^jf_*(\omega_X^{\otimes m})$ is not always a GV-sheaf
when $\kappa(X) = - \infty$. 
On the other hand, the following proposition holds:

\begin{prop}\label{prop4.8}
Let $X$ be a smooth projective surface and $f: X \to A$ be a morphism 
to an abelian variety. 
Assume that $\kappa(X) \geq 0$. 
Then $R^jf_*(\omega_X^{\otimes m})$ is a GV-sheaf for any $j$ and $m$.
\end{prop}

\begin{proof}
This is obvious when $\dim f(X) =0$. 
So we may assume that $\dim f(X) \geq 1$. 
Fix a positive integer $m$. 
$f_*(\omega_X^{\otimes m})$ is a GV-sheaf by Theorem \ref{thm4.1}. 
$R^2f_*(\omega_X^{\otimes m}) =0$ since all fibers of $f$ have dimension less than 2. 
So it is sufficient to show that $R^1f_*(\omega_X^{\otimes m})$ is a GV-sheaf. 

Since $\kappa(X) \geq 0$, we can take a series of contractions of $(-1)$-curves 
$\varepsilon :X \to X'$ such that the canonical bundle $\omega_{X'}$ of $X'$ is semi-ample. 
Then the exceptional curves of $\varepsilon$ are also contracted by $f$. 
Hence we obtain a morphism $f':X' \to A$ such that $f' \circ \varepsilon = f$. 
Consider the spectral sequence 
$$ E^{i,j}_2= R^if'_*R^j\varepsilon_*(\omega_X^{\otimes m})
\Rightarrow R^{i+j}f_*(\omega_X^{\otimes m}).$$
Then $E^{i,j}_2 =0$ except that $(i,j)=(0,0), (1,0), (0,1)$. 
So we have 
\begin{align*}
R^1f_*(\omega_X^{\otimes m}) 
&\cong R^1f'_* \varepsilon_*(\omega_X^{\otimes m}) 
\oplus f'_* R^1\varepsilon_*(\omega_X^{\otimes m})
\\ &\cong R^1f'_*(\omega_{X'}^{\otimes m})
\oplus f'_* R^1\varepsilon_*(\omega_X^{\otimes m}).
\end{align*}
Take any ample line bundle $L$ on $A$. 
We have 
$$H^k(A, R^1f'_*(\omega_{X'}^{\otimes m}) \otimes L)=0\ \mathrm{for}\ k>0$$ 
by the Ambro--Fujino vanishing theorem 
since $\omega_{X'}$ is semi-ample (see \cite[Theorem 3.2]{Amb} and 
\cite[Theorem 6.3]{Fuj1}). 
It is obvious that 
$$H^k(A, f'_* R^1\varepsilon_*(\omega_X^{\otimes m}) \otimes L)=0\ \mathrm{for}\ k>0$$ 
since $\dim \Supp R^1\varepsilon_*(\omega_X^{\otimes m}) = 0$. 
Thus it follows that 
$$H^k(A, R^1f_*(\omega_X^{\otimes m}) \otimes L)=0\ \mathrm{for}\ k>0.$$

We show by using Theorem \ref{thm4.3} that 
$R^1f_*(\omega_X^{\otimes m})$ is a GV-sheaf. 
Let $\phi:B \to A$ be an isogeny and $L$ be an ample line bundle on $B$. 
Set $Y = X \times_A B$. 
Let $\psi:Y \to X$ and $g:Y \to B$ be the natural morphisms. 
Then $\psi$ is \'etale and $\kappa(Y) \geq 0$. 
By base change, 
$$\phi^*R^1f_*(\omega_X^{\otimes m}) \cong R^1g_*\psi^*(\omega_X^{\otimes m}) 
\cong R^1g_*(\omega_Y^{\otimes m}).$$
Hence 
$$H^k(B, \phi^*R^1f_*(\omega_X^{\otimes m}) \otimes L) 
\cong H^k(B,R^1g_*(\omega_Y^{\otimes m}) \otimes L).$$
Here $H^k(B,R^1g_*(\omega_Y^{\otimes m}) \otimes L)=0$ for $k>0$ 
by the above argument. 
Therefore $H^k(B, \phi^*R^1f_*(\omega_X^{\otimes m}) \otimes L)=0$
for $k>0$. 
Then it follows by Theorem \ref{thm4.3} that 
$R^1f_*(\omega_X^{\otimes m})$ is a GV-sheaf. 
\end{proof}

Taking the above proposition into account, 
we modify the question of Popa--Schnell as follows:

\begin{ques}\label{ques4.9}
Let $X$ be a smooth projective variety of dimension $n \geq 3$ 
and $f: X \to A$ be a morphism to an abelian variety. 
Assume that $\kappa(X) \geq 0$. 
Then is $R^jf_*(\omega_X^{\otimes m})$ a GV-sheaf for any $j$ and $m$?
\end{ques}


\end{document}